\documentclass[a4paper,12pt,oneside,reqno]{amsart} 

\usepackage[utf8x]{inputenc}
\usepackage[english]{babel}
\usepackage{amsmath,amsthm,amssymb,amsfonts,amscd,amsbsy,amsxtra}
\usepackage{ucs}
\usepackage{geometry}
\usepackage{graphicx}
\usepackage{float}
\usepackage{wrapfig}

\usepackage{paralist}

\usepackage{cite}
\usepackage{url}
\usepackage{hyperref}

 \newtheorem{thm}{Theorem}[section]
 \newtheorem{cor}[thm]{Corollary}
 
 \newtheorem{prop}[thm]{Proposition}
 \theoremstyle{definition}
 \newtheorem{defn}[thm]{Definition}
 \theoremstyle{remark}
 \newtheorem{rem}[thm]{Remark}
 
 \numberwithin{equation}{section}
 
\newtheorem{theorem}{Theorem}[section]

\newtheorem{example}[theorem]{Example}

\DeclareMathOperator{\supp}{supp}
\DeclareMathOperator*{\esssup}{ess\,sup}

\begin{document}

\title[]{
Sobolev space of functions valued in a monotone Banach family
}

\author{Nikita Evseev}

\address{
Institute for Advanced Study in Mathematics, Harbin Institute of Technology, 150001, Harbin, People’s Republic of China\\
}

\address{%
Sobolev Institute of Mathematics\\
4 Acad. Koptyug avenue\\
630090 Novosibirsk\\
Russia}
\address{
Novosibirsk State University, 630090 Novosibirsk, Russia\\
}

\email{evseev@math.nsc.ru}

\author{Alexander Menovschikov}

\address{%
Sobolev Institute of Mathematics\\
4 Acad. Koptyug avenue\\
630090 Novosibirsk\\
Russia}
\address{
Novosibirsk State University, 630090 Novosibirsk, Russia\\
}

\email{menovschikov@math.nsc.ru}

\thanks{The work was supported by Russian Foundation for Basi Resear h (proje t no. 18-31-00089)}


\date{}









\begin{abstract}
We apply the metrical approach to Sobolev spaces, which arise in various evolution PDEs. 
Functions from those spaces are defined on an interval and take values in a family of Banach spaces.
In this case we adapt the definition of Newtonian spaces.
For a monotone family, we show the existence of weak derivative, obtain an isomorphism to the standard Sobolev space, and provide some scalar characteristics.
\end{abstract}

\maketitle




\section{Introduction}

Various applied problems in biology, materials science, mechanics, etc, involve PDEs with solution spaces with internal structure that changes over time. 
As examples, we review some of the recent researches: \cite{BPS01} by M.~L. {Bernardi}, G.~A. {Pozzi}, and G.~{Savar\'e}, \cite{Paronetto2013} by F. Paronetto (equations on non-cylindrical domains), \cite{VM14} by M. Vierling, \cite{AES15-2} by A. Alphonse, C. M. Elliott, and B. Stinner (equations on evolving hypersurfaces),
\cite{MB08} by S.~{Meier} and M.~{B\"ohm}, \cite{ET2012} by J. Escher and D. Treutler (modeling of processes in a porous medium). 
Common to all of the above examples is that solution spaces could be represented as sets of functions valued in a family of Banach spaces. 
However, different problems impose different requirements on the relations between spaces within families, for example, the existence of isomorphisms, embeddings, bounded operators and so on.

In this article, we consider Sobolev spaces associated with the above problems from the point of view of metric analysis. 
Although the family of Banach spaces cannot always be represented as a metric space, the metric definition of Sobolev classes remains meaningful. 
Such a point of view on the studied spaces will make it possible to apply more universal and well-developed methods. In the 90s, several authors (L.~Ambrosio \cite{Ambrosio1990}, N.J.~Korevaar and R.M.~Schoen \cite{KorSch1993}, Yu.G.~Reshetnyak \cite{Reshetnyak97} and A.~Ranjbar-Motlagh \cite{Ranjbar98}) 
introduced and studied Sobolev spaces consisting of functions taking values in metric spaces. 
The case of functions defined on a non-Euclidean space is described by P.~Haj{\l}asz \cite{Hajlasz96} and J.~Heinonen, P.~Koskela, N.~Shanmugalingam, J.T.~Tyson \cite{HKST01}. 
In \cite{HKST01} it was shown that all previously developed approaches are equivalent. 
For a detailed treatment and for references to the literature on
the subject one may refer to books \cite{H2001} by J.~{Heinonen} and \cite{HP2000} by P.~{Haj{\l}asz} and P.~{Koskela}.

For our purposes, we adapt the following definition of Sobolev space (Newtonian spaces, for real-valued case see \cite{Shanmugalingam2000}, and \cite{HKST01} for Banach-valued case). 
\textit{Function $f:(M,\rho)\to(N,d)$ from the space $L^p(M;N)$ belongs to $W^{1,p}(M;N)$, if there exists scalar function $g\in L^p(M)$ such that
\begin{equation}\label{eqNewton}
d\big(f(\gamma(a)), f(\gamma(b))\big) \leq \sup\limits_{\gamma}\int_\gamma g\,d\sigma.
\end{equation} 
}
On the one hand, we have all the necessary objects to adapt this definition. 
On the other hand, in metric case, it allows us to introduce the concept of an upper gradient, establish the Poincar\'{e} inequality, show that a superposition with a Lipschitz function also is a Sobolev function.

The evolution structure of a specific problem could be described with the help of the following objects. 
Let $\{X_t\}_{t\in (0,T)}$ be a family of Banach spaces, $(0,T)\subset\mathbb R$, and suppose
that there is a set of operators $P(s,t):X_s\to X_t$ for $t\geq s$.
We consider functions with the property that $f(t)\in X_t$. Then inequality \eqref{eqNewton} turns to 
$$
\|f(t) - P(s,t)f(s)\|_{t} \leq \int_{s}^{t} g(\tau)\, d\tau, 
$$
and defines the space $W^{1,p}((0,T); \{X_t\})$.

The first natural question arises from this definition is: what is the meaning of the function $g(t)$? 
In the case of a monotone family of reflexive spaces, the answer to this question is given by theorem \ref{thm:main1}. Namely, under such assumptions, we can explicitly construct the weak derivative and show that its norm coincides with the smallest upper gradient of the original function.

In section \ref{Isomorphism} we establish the connection of the introduced space to the standard case. 
More precisely, suppose that there is a set of local isomorphisms $\Phi_t: X_t \to Y$. We are interested if there exists a global isomorphism between Sobolev spaces $W^{1,p}((0,T), \{X_t\})$ and $W^{1,p}((0,T), Y)$. 
Due to theorem \ref{theorem:main2}, the necessary and sufficient conditions for the existence of such an isomorphism are the close interconnection of $\Phi_t$ and the transition operators $P(s,t)$.

In section \ref{Scalar}, as an example, we give a comparison of our approach with that of Yu.G. Reshetnyak 
(recall that for metric spaces both came up with the same results). 
It turns out that Reshetnyak's approach does not allow taking into account the internal structure of the spaces under consideration.

\section{Sobolev space $W^{1,p}((0,T); \{X_t\})$}\label{Sobolev-ineq}
In this section, we give the definition of the main object - Sobolev functions valued in the family of Banach spaces. 
We also provide examples of families on which our methods can be applied.

\subsection{Monotone family $\{X_t\}$ as a vector space}

Let $V$ be a vector space, $(0,T)$ be an interval (not necessarily bounded)
equipped with the Lebesgue measure,
and $\{||\cdot||_{t}\}_{t\in (0,T)}$ be a family of semi-norms on $V$.
We will assume that for each $v\in V$ function $N(t,v) = \|v\|_t$
is non-increasing:
\begin{equation}\label{eq:A2}
N(t_1,v)\geq N(t_2,v), \text{ if } t_1\leq t_2. 
\end{equation}
Define Banach space $X_t$ to be a completion $V/\ker(\|\cdot\|_{t})$ with respect to $\|\cdot\|_{t}$.
Then $\{X_t\}$ is said to be a monotone family of Banach spaces (or a monotone Banach family).

For $t_1\leq t_2$ define operators $P(t_1,t_2):X_{t_1}\to X_{t_2}$ in the following way.
If $x\in X_{t_1}$ then there is a sequence $\{v_k\} \subset V$ such that it is a Cauchy sequence 
with respect to $\|\cdot\|_{t_1}$ and $x$ is its limit in $X_{t_1}$. 
Due to \eqref{eq:A2} $\{v_k\}$ is a Cauchy sequence with respect to $\|\cdot\|_{t_2}$,
thus there is its limit  $\tilde x$ in $X_{t_2}$. 
Put $P(t_1,t_2)x = \tilde x$. 
Now we can construct an addition. If $x_i\in X_{t_i}$ then
\begin{equation}\label{eq:addition}
x_1+x_2 := \begin{cases} P(t_1,t_2)x_1 + x_2, \text{ if } t_1\leq t_2,\\
x_1 + P(t_2,t_1)x_2, \text{ if } t_1 > t_2.
\end{cases}
\end{equation}
One can show that this operation satisfy  the associative and commutative  properties.
So the pare $(\bigcup_{t}X_t, +)$ is a vector space over $\mathbb R$.

\subsection{$L^p$-direct integral of Banach spaces}
We deal with the $L^p$-spaces of mappings
$f:(0,T)\to\bigcup_{t}X_t$ with the property that 
$f(t)\in X_t$ for each $t\in (0,T)$
(in other words, $f$ is a section of $\{X_t\}$ ).
To make this treatment rigorous, we apply the concept of direct integral of Banach spaces.
A brief account of the theory of direct integral is given below 
(for detailed presentation see \cite{HLR91} and \cite{JR17}).

Note that monotonicity condition \eqref{eq:A2} implies that $\{X_t\}$ is a measurable family 
of Banach spaces over $((0,T), dt, V)$ in the sense of \cite[Section 6.1]{HLR91}. 
\begin{defn}
A \textit{simple section} is a section $f$ for which there exist $n\in\mathbb N$, $v_1, \dots, v_n\in V$, and measurable sets $A_1, \dots, A_n\subset T$ such that $f(t) = \sum_{k=1}^n\chi_{A_k}\cdot v_k$ 
for all $t\in (T,0)$.
\end{defn}

\begin{defn}
A section $f$ of $\{X_t\}_{t \in (0,T)}$ is said to be \textit{measurable} if there exists a sequence of simple sections $\{f_k\}_{k\in\mathbb N}$ such that for almost all $t\in (0,T)$, $f_k(t)\to f(t)$ in $X_t$ as $k\to\infty$.
\end{defn}

The space of all equivalence classes of such measurable sections is a \textit{direct integral} 
$\int_{(0,T)}^{\oplus}X_t\,dt$ of a monotone family of Banach spaces $\{X_t\}_{t \in (0,T)}$
We will denote this space as $L^0((0,T); \{X_t\})$.

Note that for a measurable section $f$ the function $t\mapsto\|f(t)\|_{t}$ is measurable in the usual sense.
For every $p\in[1,\infty]$ the space 
$L^p((0,T); \{X_t\}) = \left(\int_{(0,T)}^{\oplus}X_t\,dt\right)_{L^p}$ (\textit{$L^p$-direct integral})
is defined 
as a space of all measurable sections $f$ such that the function $t\mapsto\|f(t)\|_{t}$ belongs to $L^p((0,T))$.
In this case 
$$
\|f\|_{L^p((0,T); \{X_t\})}
:=  
\begin{cases}
\left(\int_{0}^T\|f(t)\|^p_{t}\,dt\right)^{\frac{1}{p}}, &\text{ if } p<\infty,\\
\esssup\limits_{(0,T)} |f(t)|, &\text{ if } p=\infty
\end{cases}
$$
determines the norm on $L^p((0,T); \{X_t\})$.
\begin{prop}[{\cite[Proposition 3.2]{JR17}}]\label{lp-banach}
The space $L^p((0,T); \{X_t\})$ is a Banach space for all $1\leq p <\infty$.
\end{prop}

Define \textit{sectional weak convergence} $f_{n}(t)\rightharpoonup f(t)$ for a.e. 
$t\in (0,T)$
as $\langle b'(t), f_n(t) \rangle_{t} \to \langle b'(t), f(t) \rangle_{t}$ a.e. 
for all $b'(t)\in X'_{t}$. 
Then applying a standard technique one can prove the following proposition.

\begin{prop}\label{prop:weak-lp}
Let $f_n\in L^p((0,T); \{X_t\})$, $\|f_n\|_{L^p((0,T); \{X_t\})}\leq C<\infty$ and 
$f_{n}(t)\rightharpoonup f(t)$ for a.e. $t\in (0,T)$.
Then $f\in L^p((0,T); \{X_t\})$ and $\|f\|_{L^p((0,T); \{X_t\})}\leq C$.
\end{prop}

\subsection{Sobolev space $W^{1,p}((0,T); \{X_t\})$}
As it was pointed out in the introduction, we adapt \eqref{eqNewton} and obtain the definition of Newtonian space for functions valued in a monotone family.

\begin{defn}\label{defn:W1}
A measurable section $u:T\to\bigcup_{t}X_t$ 
is said to be in the Sobolev space $W^{1,p}((0,T); \{X_t\})$ 
if $u\in L^p((0,T); \{X_t\})$ and if there exist
a function $g\in L^p((0,T))$ so that 
\begin{equation}\label{eq:criteria}
\|u(t) - P(s,t)u(s)\|_{t} \leq \int_{s}^{t} g(\tau)\, d\tau,
\end{equation}
for almost all $s, t\in (0,T)$, $s\leq t$. 
\end{defn}
A function $g$ satisfying \eqref{eq:criteria} is called 
a \textit{$p$-integrable upper gradient} of $u$ (or just upper gradient). 
If $u$ is a function in $W^{1,p}((0,T); \{X_t\})$,
let
$$
\|u\|_{W^{1,p}((0,T);\{X_t\})} := \|u\|_{L^{p}((0,T); \{X_t\})} + 
\inf\limits_{g} \|g\|_{L^{p}((0,T))},   
$$
where the infimum is taken over all upper $p$-integrable upper gradients $g$ of $u$.
It is assumed that $W^{1,p}((0,T); \{X_t\})$ consists of equivalence classes of functions,
where $f_1\sim f_2$ means $\|f_1 - f_2\|_{W^{1,p}((0,T);\{X_t\})} = 0$. 
Thus $W^{1,p}((0,T); \{X_t\})$ is a normed space, and it is a Banach space 
(see theorem \ref{theorem:W-Banach}).

\begin{rem}
It is worth noting that in definition \ref{defn:W1} we do not use the monotonicity of the family and for its correctness it is enough to have a family of operators $P(s,t)$ that produce a vector space structure on $\bigcup_{t}X_t$. 
However, in the next section, to construct a weak derivative, the monotonicity of the family will be essential.
\end{rem}

\subsection{Examples}
Here we will provide some more or less explicit examples. 
\begin{example}\label{ex1}
Let $\{\Omega_t\}_{t\in(0,T)}$ be a non-increasing family of measurable sets: $\Omega_t\subset\Omega_s$ if $s<t$. Let $\Omega_0 = \bigcup_{t}\Omega_t$.
As a core vector space $V$ choose the space of step functions on $\Omega_0$ and
define semi-norms $N(t,v) = \|v\|_{L^q(\Omega_t)}$.
Then family $\{L^q(\Omega_t)\}$ is monotone and operators $P(s,t):L^q(\Omega_s)\to L^q(\Omega_t)$ are restriction operators: $P(s,t)f = f_{|\Omega_t}$ for $f\in L^q(\Omega_s)$. 


\begin{figure}[h!]
\centering
\begin{minipage}{.5\textwidth}
  \centering
  \includegraphics[width=.9\linewidth]{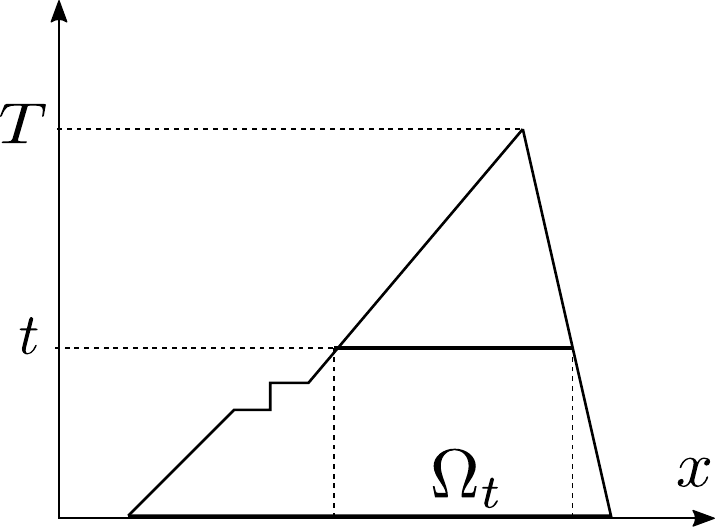}
  \caption{Example \ref{ex1}}
  \label{fig1}
\end{minipage}%
\begin{minipage}{.5\textwidth}
  \centering
  \includegraphics[width=.9\linewidth]{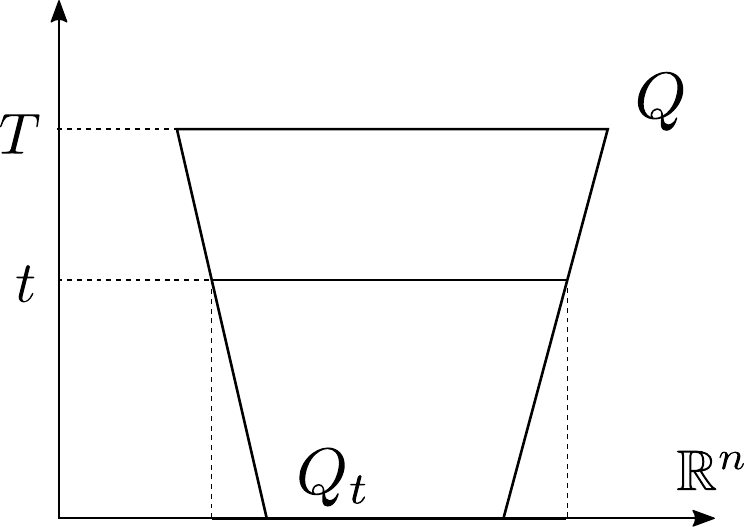}
  \caption{Example \ref{ex-Hilbert}}
  \label{fig2}
\end{minipage}
\end{figure}

Moreover, element $u \in L^p((0,T); \{L^q(\Omega_t)\})$ could be represented as a function $u(t,x)$
belonging to mixed norm Lebesgue space $L^{p,q}(\Omega)$, where $\Omega = \bigcup_{t}t\times\Omega_t$.   
\end{example}

\begin{example}[Evolving spaces]

In \cite{AES15} an abstract framework has been developed for
treating parabolic PDEs on evolving Hilbert spaces. Some applications of this method are in \cite{AES15-2, AE16, Djurdjevac17}. 

Here we compare the compatibility property from \cite{AES15}  and with construction.
As in \cite{AES15}, given a family of Hilbert spaces $\{X(t)\}_{t\in[0,T]}$
and a family of linear homeomorphisms $\phi_t:X(0)\to X(t)$.
For all $v\in X(0)$ and $u\in X(t)$ the following conditions are assumed

\begin{enumerate}[(C1)]
\item $\|\phi_tv\|_{X(t)} \leq C\|v\|_{X(0)}$,
\item $\|\phi_t^{-1}u\|_{X(0)} \leq C\|u\|_{X(t)}$,
\item $t\mapsto\|\phi_tv\|_{X(t)}$ is continuous.
\end{enumerate}

On the one hand, we can not formulate this structure in our settings straightway.
On the other hand, we can construct another family of Banach spaces such that $L^2$-direct integral of this family is isomorphic to space $L^2_X$  
form \cite[Definition 2.7]{AES15}. 
Set $V = X(0)$,
$N(t,v) = \|\phi_tv\|_{X(t)}$, 
and $P(s,t) = \phi_t\phi^{-1}_s$. 
Then condition (C3) implies that $\{(X(t), \|\cdot\|_{X(t)})\}$
is a measurable family.
\end{example}

\begin{example}[Composition operator]\label{ex-compop}
Let $\Omega_0$ be a domain in $\mathbb{R}^n$.
Let us consider the Sobolev space $W^{1,q}(\Omega_0; \mathbb{R})$ as a core vector space $V$.
We are going to construct a monotone family of Banach spaces which is generated by a family of quasi-isometric mappings $\varphi(t,\cdot):\Omega_0 \to \Omega_t$.
Each of these mapping induces isomorphism $C_{\varphi(t,\cdot)}:W^{1,q}(\Omega_t)\to W^{1,q}(\Omega_0)$ by the composition rule (\cite[Theorem 4]{V2005}).
Define $N(t,v) = \|C^{-1}_{\varphi(t,\cdot)}v\|_{W^{1,q}(\Omega_t)}$, and choose mappings $\varphi(t,\cdot)$ such that the family of norms is monotone.
As a result, we obtain spaces $X_t$, which consist of functions from $W^{1,q}(\Omega_0; \mathbb{R})$ and endowed with the norm $\|C^{-1}_{\varphi(t,\cdot)}v\|_{W^{1,q}(\Omega_t)}$.
Thus we can define the Sobolev space $W^{1,p}((0,T); \{X_t\})$ over this family in the sense of \ref{defn:W1}.
In that case, operators $P(s,t):W^{1,q}(\Omega_s)\to W^{1,q}(\Omega_t)$ are composition operators induced by $\varphi(s,\cdot)\circ\varphi^{-1}(t,\cdot):\Omega_t\to\Omega_s$.

\end{example}

\begin{example}[Monotone family of Hilbert spaces]\label{ex-Hilbert}
The next example is taken from \cite{BPS01}. 
In that work Cauchy-Dirichlet problems for linear Schrodinger-type equations in non-cylindrical domains are studied. 
Note that the monotonicity condition is important for their considerations. 

Let $Q\subset\mathbb R^n\times(0,T)$ be an open set and its sections 
$
Q_t = \{x\in R^n : (x,t)\in Q\}
$
be a non-decreasing family. Define $Q_T = \bigcup_{t} Q_t$ (see figure \ref{fig2}).

Let $V=H^1_0(Q_T)$ and $V_t$ be a completion of 
$\{v\in C_0^{\infty}(Q_T) :\supp v \subset Q_t \}$ with respect to the norm $\|\cdot\|_{H^1_0(Q_T)}$.
Let $\pi(t):H^1_0(\Omega_T)\to V_t$ be an orthogonal projector, then 
define $N(t, v) = \|\pi(t) v\|_{H^1_0(Q_T)}$. 
At the same time, $V_t$ is a completion $V/\ker(N(t,\cdot))$ with respect to $N(t, \cdot)$. 
Therefore operators $P(s,t)$ are just trivial extensions to $Q_t$. 
\end{example}

\section{Calculus of $\{X_t\}$-valued functions}
\label{section:calculus}
In this section and below we will denote as $t_1 \vee t_2$ the maximum of this numbers and as $t_1 \wedge t_2$ the minimum.
\subsection{Limit and continuity}

Here we introduce the concepts of limit, continuity and differentiability 
for $\{X_t\}$-valued functions. 
Due to the addition defined in \eqref{eq:addition} all basic properties are preserved for those notions.

\begin{defn}
An element $\xi\in X_{t_0}$ is said to be the limit of a section $f(t)$ for $t \to t_0$:
$\lim\limits_{t \to t_0} f(t) = \xi$, if
$
\|f(t) - \xi \|_{t_0 \vee t} \to 0  \text{ as } t \to t_0.
$
\end{defn}

\begin{defn}\label{def:continuity}
A section $f(t)$ is continuous at the point $t_0 \in (0,T)$, if
$$
\lim\limits_{t \to t_0} f(t) = f(t_0) \in X_{t_0}.
$$
\end{defn}
By $C(J; \{X_t\})$ we will denote the set of continuous functions at every point of $J\subset (0,T)$.

\begin{defn}[Fr\'echet derivative]\label{def:der}
A section $f(t)$ is differentiable at $t_0\in (0,T)$ if there exists $l_{t_0}\in X_{t_0}$ and, for every $\varepsilon>0$, exists $\delta>0$ such that
$$
\|f(t_0 + h) - f(t_0) - l_{t_0}h \|_{t_0 \vee (t_0 +h) } \leq \varepsilon|h|
$$
for all $|h|\leq\delta$. In what follows we denote $l_{t_0} = \dfrac{df}{dt}(t_0)$. 
\end{defn}

\begin{defn}
Let $[a,b]\subset (0,T)$ be a bounded interval.
A function $f:[a,b]\to\bigcup_{t\in[a,b]}X_t$ is said to be \textit{absolutely continuous}, if for any $\varepsilon>0$ there exist $\delta>0$ such that
$\sum_{i=1}^n\|f(b_i)-f(a_i)\|_{b_i} \leq \varepsilon$ 
for any collection of disjoint intervals $\{[a_i,b_i]\}\subset [a,b]$ such that
$\sum_{i=1}^n(b_i-a_i) \leq \delta$.
\end{defn}
A function $f:J\to\bigcup_{t\in J}X_t$ is said to be \textit{locally absolutely continuous} on a set $J$, if it is absolutely continuous for any interval $[a,b]\subset J$.

\subsection{Local Bochner integral}
Let there be given a simple function 
$$
s(t) = \sum\limits_{i=1}^{m}v_i\chi_{E_i},
$$
where $v_i\in V$ and $\{E_i\}\subset (0,T)$ is a disjointed collection of measurable sets of finite measure.
Then the integral is defined as
$$
\int^T_0 s(t)\, dt = \sum\limits_{i=1}^{m}v_i|E_i|.
$$
Now we introduce the notion of local integrability for a measurable section. 
\begin{defn}\label{def:integral}
A measurable function $f\in L^0((0,T);\{X_t\})$ is called \textit{locally integrable}, if for every compact set $J\subset (0,T)$ there exist a sequence of simple functions $\{s_k(t)\}$ such that 
\begin{equation}\label{int0}
\int_0^T\|\chi_J(t)\cdot f(t) - s_k(t)\|_{t}\, dt \to 0
\quad \text{ for }  k\to\infty.
\end{equation}
\end{defn}
Note that if the function $f$ is locally integrable, then for the sequence from definition \ref{def:integral} we have  
$
\lim\limits_{k\to\infty}\int_0^T\|s_k(t)\|_t\, dt  = \int_0^T\|\chi_J(t)\cdot f(t)\|_t \, dt.
$

\begin{prop}\label{int1}
Let $f\in L^0((0,T);\{X_t\})$ be locally integrable.
Then, for every compact set $J\subset (0,T)$, there exists $x\in X_{t^*}$, $t^* = \sup J$, such that, 
for any sequence of simple functions $s_k(t)$ with the property  
$\int_0^T\|\chi_J(t)\cdot f(t) - s_k(t)\|_{t}\, dt \to 0$, the following convergence holds
$$
\bigg\|\int_0^T s_k(t)\, dt  - x \bigg\|_{t^*}\to 0 \quad \text{ for } k\to\infty.  
$$
\end{prop}
\begin{proof} 1) Let $J\subset (0,T)$ be a compact set and $t^* = \sup J$.
Next we prove that the sequence $\int_0^T s_k(t)\, dt$ is fundamental in $X_{t^*}$.
\begin{multline*}
\bigg\|\int_0^T s_k(t)\, dt  - \int_0^T s_m(t)\, dt \bigg\|_{t^*}
\leq \int_0^T \| s_k(t) - s_m(t)\|_{t^*} \, dt\\
\leq \int_0^T \| \chi_J(t)\cdot f(t) - s_m(t)\|_{t} \, dt
+ \int_0^T \| \chi_J(t)\cdot f(t) - s_k(t)\|_{t} \, dt
\to 0
\end{multline*}
for $k,m\to\infty$. Hence, there is $x\in X_{t^*}$ such that 
$\lim\limits_{k\to\infty} \int_0^T s_k(t)\, dt = x$  in $ X_{t^*}$.

2) If, for another sequence of simple functions $r_k(t)$, it is true that 
$\int_0^T\|\chi_J(t)\cdot f(t) - r_k(t)\|_{t}\, dt \to 0$ for $k\to\infty$, then
\begin{multline*}
\bigg\|\int_0^T r_k(t)\, dt  - x \bigg\|_{t^*}
\leq \int_0^T \| \chi_J(t)\cdot f(t) - r_k(t)\|_{t} \, dt\\
+ \int_0^T \| \chi_J(t)\cdot f(t) - s_k(t)\|_{t} \, dt
+\bigg\|\int_0^T s_k(t)\, dt  - x \bigg\|_{t^*}
\to 0.
\end{multline*}  
\end{proof}

\begin{defn}
The integral over compact set $J\subset (0,T)$ of a locally integrable function is an element $x$ from proposition \ref{int1}, i.~e.
\begin{equation}\label{eq:integral}
\int_J f\, dt = \int_0^T \chi_J(t)\cdot f(t)\, dt
: = \lim\limits_{k\to\infty} \int_0^T s_k(t)\, dt = x \in X_{t^*},
\end{equation} 
where $t^* = \sup J$.
\end{defn}
We say that $f\in L^1_{loc}((0,T);\{X_t\})$, if $\|f(t)\|_t\in L^1_{loc}((0,T))$.
In essence, the introduced integral is a local version of the Bochner integral. 
Integral \eqref{eq:integral} has the usual additivity property.
Thus, for two intervals $(a,b_1)\subset(a,b_2)\subset (0,T)$, we have the equality
$$
\int^{b_1}_a f(t)\, dt - \int^{b_2}_a f(t)\, dt
= -\int^{b_2}_{b_1} f(t)\, dt
$$

Let us prove an analog of Bochner's theorem:
\begin{thm}\label{theorem:Bochner}
A measurable function $f\in L^0((0,T);\{X_t\})$ is locally integrable if and only if $f\in L^1_{loc}((0,T);\{X_t\})$.
For any compact set $J\subset (0,T)$ the estimate holds
$$
\bigg\|\int_J f(t)\, dt\bigg\|_{t^*}
\leq \int_J \|f(t)\|_{t^*}\, dt \,
\bigg(\leq \int_J \|f(t)\|_{t}\, dt \bigg),
$$
where $t^* = \sup J$.
\end{thm}
\begin{proof}
Let $f(t)$ be locally integrable. 
For arbitrary compact set $J\subset (0,T)$ there is a sequence of simple functions $s_k(t)$ such that convergence \eqref{int0} holds. 
Then
$$
\int_J \|f(t)\|_{t}\, dt 
\leq \int_0^T\|\chi_J(t)\cdot f(t) - s_k(t)\|_{t}\, dt
+ \int_0^T\|s_k(t)\|_{t}\, dt.
$$
The right hand of inequality is finite, thus $f\in L^1_{loc}((0,T);\{X_t\})$.

Now let $f\in L^1_{loc}((0,T);\{X_t\})$. 
Consider a sequence of simple functions $\{s_k(t)\}$ such that
$\|f(t) - s_k(t)\|_{t} \to 0$ a.e. 
Let $J\subset (0,T)$ be a compact set. Define a new sequence of simple functions
$$
r_k(t) =
\begin{cases}
s_k(t), &\text{ if } \|s_k(t)\|_t\leq 2\|f(t)\|_t \text{ and } t\in J,\\
0, &\text{ otherwise. }
\end{cases}
$$

Then $\|\chi_J(t)\cdot f(t) - r_k(t)\|_{t} \to 0$ a.e.
Further,  $\|\chi_J(t)\cdot f(t) - r_k(t)\|_{t} 
\leq \|r_k(t)\|_{t} + \|\chi_J(t)\cdot f(t)\|_{t}
\leq 3\chi_J(t)\cdot\|f(t)\|_t$ a.e.
So $\|\chi_J(t)\cdot f(t) - r_k(t)\|_{t}$ has an integrable majorant, and by the Lebesgue theorem we obtain
$$
\int_0^T\|\chi_J(t)\cdot f(t) - r_k(t)\|_{t}\, dt \to 0
\quad \text{ for }  k\to\infty.
$$ 
Hence, the function $f(t)$ is locally integrable. 
\end{proof}

\begin{rem}\label{rem:3.9}
Let $f(t)$ be locally integrable.
Then for any compact set $J\subset (0,T)$ a sequence of simple functions $\{s_k(t)\}$ 
as in definition \ref{def:integral} can be chosen so that $\supp s_k(t)\subset J$ and 
$\|s_k(t)\|_t \leq 2\|f(t)\|_{t}$.
\end{rem}

\begin{prop}[Lebesgue's Differentiation Theorem]\label{prop:lebesgue-th}
  Let $f \in L^1_{loc}((0,T);\{X_t\})$. Then, for $h>0$,
  \begin{equation}\label{lebesgue1}
    \lim\limits_{h \to 0} \frac{1}{h} \int_{t-h}^{t} \| f(s) - f(t) \|_t \, ds = 0.
  \end{equation}
  In particular,
  \begin{equation}\label{lebesgue2}
    f(t) = \lim\limits_{h \to 0} \frac{1}{h} \int_{t-h}^{t} f(s) \, ds.
  \end{equation}
\end{prop}

\begin{proof}
  Let us choose a sequence $\{x_n(t)\}_{n \in \mathbb{N}}$, that is dense in $X_t$. For every $n \in \mathbb{N}$ we consider real-valued function $\| f(t) - x_n(t) \|_t$. Applying the real-valued Lebesgue's differentiation theorem, we can find a set $E_n \subset \mathbb{R}$ for all $n \in \mathbb{N}$ such that
  $$
    \| f(t) - x_n(t) \|_t = \lim\limits_{h \to 0} \frac{1}{h} \int_{t-h}^{t} \| f(s) - x_n(t) \|_t \, ds
  $$  
  for all $t \notin E_n$. Further, for arbitrary $\varepsilon > 0$ and $t \notin \bigcup_{n \in \mathbb{N}} E_n$, there is a number $n$ such that $\| f(t) - x_n(t) \|_t < \frac{\varepsilon}{2}$. 
Using the inequality from remark \ref{rem:3.9}, we have
  \begin{multline*}
    0 \leq \lim\limits_{h \to 0} \frac{1}{h} \int_{t-h}^{t} \| f(s) - f(t) \|_t \, ds\\ 
	\leq \lim\limits_{h \to 0} \frac{1}{h} \int_{t-h}^{t} \| f(s) - x_n(t) \|_t + \| x_n(t) - f(t) \|_t \, ds 
	= 2 \| f(t) - x_n(t) \|_t < \varepsilon.
  \end{multline*}  
  
  Due to arbitrariness of choosing $\varepsilon$, statement \eqref{lebesgue1} of the theorem is proven. The second assertion follows from the first and analogue of Bochner theorem \ref{theorem:Bochner}.
\end{proof}

\begin{prop}\label{prop:Mh}
  Let $f$ belong to $L^p(\mathbb{R};\{X_t\})$, $1 \leq p < \infty$. For every $h > 0$ we define a new function $M_h f$ as
  $$
    M_h f(t) = \frac{1}{h}\int_{t-h}^{t} f(s) \, ds.
  $$
  Then $M_h f$ belongs to $L^p(\mathbb{R};\{X_t\}) \cap C(\mathbb{R};\{X_t\})$ and $\lim\limits_{h \to 0} M_h f = f$ a.e. and in $L^p(\mathbb{R};\{X_t\})$. 
\end{prop}

\begin{prop}\label{int-der}
Let $g\in L^1_{loc}((0,T); \{X_t\})$ and $\omega_0\in (0,T)$. Define a function
$
f(t) = \int_{t_0}^t g(s)\, ds
$
for $t\geq t_0$. Then

1) $f\in C(\{t \leq t_0\}\cap (0,T);\{X_t\})$,

2) $f$ is locally absolutely continuous on $\{t \geq t_0\}\cap (0,T)$,

3) $\int_0^T\varphi'(t)f(t)\, dt = \int_0^T\varphi(t)g(t)\, dt$
for all $\varphi\in C_0^{\infty}(\{t\geq t_0\}\cap (0,T))$,

4) $f$ is differentiabile a.e. on $\{t\geq t_0\}\cap (0,T)$ and $\dfrac{df}{dt}(t) = g(t)$.
\end{prop}
\begin{proof}
1) According to definition \ref{def:continuity} of continuity  for any $\omega_1\in \{\omega\leq \omega_0\}\cap (0,T)$, we obtain
\begin{multline*}
\bigg\| \int_{\omega_0}^\omega g(s)\, ds - \int_{\omega_0}^{\omega_1} g(s)\, ds \bigg\|_{\omega_1\vee \omega}
= \bigg\| \int^{\omega\vee\omega_1}_{\omega\wedge\omega_1} g(s)\, ds \bigg\|_{\omega_1\vee \omega}\\
\leq \int^{\omega\vee\omega_1}_{\omega\wedge\omega_1} \|g(s)\|_s\, ds
\to 0 \quad \text{ for } \omega\to\omega_1.
\end{multline*}

2) This assertion is also verified by definition 
$$
\sum_{i=1}^n\|f(b_i)-f(a_i)\|_{b_i} 
= \sum_{i=1}^n \bigg\|\int_{a_i}^{b_i}g(s)\, ds\bigg\|_{b_i}
\leq \sum_{i=1}^n \int_{a_i}^{b_i}\|g(s)\|_{s}\, ds.
$$
Thus, the statement follows from the absolute continuity of the Lebesgue integral.

3) Let $\varphi\in C_0^{\infty}(\{\omega\leq \omega_0\}\cap (0,T))$. 
We choose $h_*>0$ so that $\supp \varphi(\omega + h_*) \subset \{\omega\leq \omega_0\}\cap (0,T)$.
By proposition \ref{prop:Mh}
\begin{multline*}
\int_0^T\varphi'(\omega)f(\omega)\, d\omega
= \int_0^T\lim\limits_{h\to0, h<h_*}\frac{\varphi(\omega+h)-\varphi(\omega)}{h}f(\omega)\, d\omega\\
=\lim\limits_{h\to0, h<h_*}\bigg(\int_0^T \frac{\varphi(\omega+h)}{h}f(\omega)\, d\omega -
\int_0^T \frac{\varphi(\omega)}{h}f(\omega)\, d\omega 
\bigg)\\
=\lim\limits_{h\to0, h<h_*}\bigg(\int_0^T \frac{\varphi(\omega)}{h}f(\omega-h)\, d\omega -
\int_0^T \frac{\varphi(\omega)}{h}f(\omega)\, d\omega 
\bigg)\\
=-\lim\limits_{h\to0, h<h_*}\bigg(\int_0^T \varphi(\omega)\frac{f(\omega-h)-f(\omega)}{h}\, d\omega
\bigg)\\
=-\lim\limits_{h\to0, h<h_*}\bigg(\int_0^T \varphi(\omega)M_hg(\omega)\, d\omega
\bigg)
= -\int_0^T \varphi(\omega)g(\omega)\, d\omega.
\end{multline*}

4) Let us verify differentiability by definition \ref{def:der}. 
Fix $\varepsilon>0$ and $\omega_1\in \{\omega\leq \omega_0\}\cap (0,T)$. 
By Lebesgue theorem \ref{prop:lebesgue-th} there is $\delta>0$ such that, for all $|h|<\delta$,
\begin{equation}
\frac{1}{h}\int_{\omega_1}^{\omega_1+h} \|g(s) - g(\omega_1)\|_{\omega_1+h}\, ds < \varepsilon.
\end{equation}

Then, for $h>0$ (the case $h<0$ can be considered similarly), we obtain
\begin{multline*}
\|f(\omega_1+h) - f(\omega_1) - hg(\omega_1)\|_{\omega_1\vee(\omega_1+h)}
= \bigg\| \int_{\omega_1}^{\omega_1+h} g(s)\, ds - hg(\omega_1) \bigg\|_{\omega_1+h}\\
\bigg\| \int_{\omega_1}^{\omega_1+h} g(s) - g(\omega_1)\, ds \bigg\|_{\omega_1+h}
\leq
|h|\frac{1}{|h|}\int_{\omega_1}^{\omega_1+h} \|g(s) - g(\omega_1)\|_{\omega_1+h}\, ds
\leq |h|\varepsilon.
\end{multline*}

\end{proof}

\section{Sobolev space $W^{1,p}((0,T); \{X_t\})$ via weak derivative}\label{Derivatives}
Here we show that for a Sobolev function from $W^{1,p}((0,T); \{X_t\})$ there exist a weak derivative, which is a section of $\{X_t\}$ and belongs $L^{p}((0,T); \{X_t\})$.
To do that we adapt the classical scheme for Banach valued functions
by using the concept of local Bochner integral, for example, see \cite{CH98}. 

\subsection{Weak derivatives}
\begin{defn}
Let $f\in L^1_{loc}((0,T); \{X_t\})$. 
A function $g\in L^1_{loc}((0,T); \{X_t\})$ is called 
\textit{a weak derivative of $f$} (the usual notation $g = f'$), if for all 
$\varphi\in C_0^{\infty}((0,T))$ the next equality holds
\begin{equation}\label{gen-der}
\bigg\|\int_0^T \varphi'(t)f(t)\, dt + \int_0^T \varphi(t)g(t)\, dt\bigg\|_{t^*}= 0,
\end{equation} 
where $t^* = \sup\{\supp\varphi(t)\}$.
\end{defn}

\begin{prop}\label{prposition:x_0}
Let $f\in L^1_{loc}((0,T); \{X_t\})$ and a weak derivative $f'(t)=0$ a.e. on interval $J\subset (0,T)$.
Then, there exist an element $x_0\in\bigcap\limits_{t\in J}X_t$ such that
$$
\int_0^T \varphi(t)f(t)\, dt = x_0\int_0^T \varphi(t)\, dt
$$
for all $\varphi\in C_0^{\infty}(J)$. In other words, $f(t) = x_0$ a.e. on $J$.
\end{prop}
\begin{proof}
Let $b\in J$. Let us choose $\vartheta\in C_0^{\infty}(J)$ such that
$b = \sup\{\supp\varphi(t)\}$ and $\int_J\vartheta(t)\, dt = 1$.
Assuming that $x_0 = \int_J\vartheta(t)f(t)\, dt\in X_b$.

For an arbitrary function $\varphi \in C_0^{\infty}(J)$ we define a new function 
$$
\psi(t) = \int_{t_0}^t(\varphi(s) -\vartheta(s)\int_0^T\varphi(\sigma)\, d\sigma)\, ds,
$$
where $t_0 = \inf\{\supp\varphi(t)\}$.
Then $\psi \in C_0^{\infty}(J)$ and 
$\psi'(t) = \varphi(t) -\vartheta(t)\int_0^T\varphi(\sigma)\, d\sigma$.
By the hypothesis of the theorem
\begin{multline*}
0 = \int_0^T\psi(t)f'(t)\, dt = \int_0^T\psi'(t)f(t)\, dt\\  
= \int_0^T\varphi(t)f(t)\, dt - \int_0^T\varphi(\sigma)\, d\sigma\cdot\int_0^T\vartheta(t)f(t)\, dt.
\end{multline*}
Therefore $\int_0^T\varphi(t)f(t)\, dt = x_0\int_0^T\varphi(\sigma)\, d\sigma$.
Due to arbitrariness of $b$, we conclude that $x_0\in\bigcap\limits_{t\in J}X_t$.
\end{proof}

\begin{prop}\label{prop:sobolev-int}
Let $1\leq p \leq \infty$ and let $u\in L^{p}((0,T); \{X_t\})$
have a weak derivative $u'\in L^{p}((0,T); \{X_t\})$.
Then the equality
$$
u(t) = u(t_0) + \int_{t_0}^{t}u'(s)\, ds
$$
holds for almost all $t_0,t\in (0,T)$, $t_0\leq t$.
Particularly,
\begin{equation}\label{eq:prop:sobolev-int}
\|u(t) - u(t_0)\|_{t} \leq  \int_{t_0}^{t}\|u'(s)\|_{t}\, ds
\leq  \int_{t_0}^{t}\|u'(s)\|_{s}\, ds.
\end{equation}
\end{prop}
\begin{proof}
Let $t_1\in (0,T)$. Define functions
$f(t) = \int_{t_1}^{t}u'(s)\, ds$
and $r(t) = u(t) - f(t)$ for all $t>t_1$.
By proposition \ref{int-der}
\begin{multline*}
\int_0^T\varphi'(t)r(t)\, dt = \int_0^T\varphi'(t)u(t)\, dt
- \int_0^T\varphi'(t)f(t)\, dt\\ 
= -\int_0^T\varphi(t)u'(t)\, dt + \int_0^T\varphi'(t)u(t)\, dt = 0 
\end{multline*}
for all $\varphi\in C_0^{\infty}(\{t_0\leq t\})$.

Due to proposition \ref{prposition:x_0} there exists $x_0\in X_{t_1}$ such that
$w(t) = x_0$ for almost all $t\geq t_1$.
Hence  $u(t) = u(t_0) + \int_{t_1}^{t}u'(s)\, ds$ 
for almost all $t \geq t_1$. 
Choose such a point $t_0\geq t_1$ which satisfies the equation. 
\end{proof}

\begin{prop}\label{prop:equivW}
Let $u\in L^p((0,T);\{X_t\})$ for some $1\leq p \leq \infty$,
then the following statements are equivalent

(i) There exists a weak derivative $u'\in L^p((0,T);\{X_t\})$.

(ii) Function $u$ locally absolutely continuous on $(0,T)$, 
differentiable a.~e. 
and the derivative $\dfrac{du}{d\omega}\in L^p((0,T);\{X_t\})$.

(iii) There exists a function $v\in L^p((0,T); \{X_t\})$ such that for a.~e. $b\in (0,T)$
and each $x'(b)\in X'_b$
the function $\psi_b(\omega) = \langle x'(b), u(\omega) \rangle_{b}$ is locally absolutely continuous and
$\psi'_b(\omega) = \langle x'(b), v(\omega) \rangle_{b}$ for a. e. $\omega\leq b$.

(iv) There exists a function $v\in L^p((0,T);\{X_t\})$ such that for a.~e. $b\in (0,T)$
and each  $e'(b)\in E'_b\subset X'_b$ ($E'_b$ is a countable dense subset of $X'_b$)
the function $\psi_b(\omega) = \langle e'(b), u(\omega) \rangle_{b}$ is locally absolutely continuous and
$\psi'_b(\omega) = \langle e'(b), v(\omega) \rangle_{b}$ for a.~e. $\omega\leq b$.
\end{prop}
\begin{proof}
(i) $\Rightarrow$ (ii) 
Thanks to proposition \ref{prop:sobolev-int} $u(\omega) = u(\omega_0) + \int_{\omega_0}^{\omega}u'(s)\, ds$ 
for almost all $\omega, \omega_0\in (0,T)$, $\omega_0\leq\omega$.
Then it follwos from proposition \ref{int-der}
that function $u$ locally absolutely continuous and differentiable a.~e. on $(0,T)$.
As well $\dfrac{du}{d\omega} = u'\in L^p((0,T);\{X_t\})$. 

(ii) $\Rightarrow$ (iii)
Note that
$$
|\langle x'(b), u(\omega_2) \rangle_{b} - \langle x'(b), u(\omega_1) \rangle_{b}|
\leq \|x'(b)\|\cdot\|u(\omega_2) - u(\omega_1) \|
$$
holds for $\omega_1\leq\omega_2\leq b$.
Which implies locally absolutely continuity of function $\psi_b(\omega)$.
Compute the derivative:
\begin{multline*}
\psi'_b(\omega) = \lim\limits_{h\to 0+}\frac{\psi_b(\omega) - \psi_b(\omega-h)}{h}\\
= \lim\limits_{h\to 0+}\langle x'(b), \frac{u(\omega) - u(\omega-h)}{h}\rangle_{b}
= \langle x'(b), \frac{du}{d\omega}(\omega)\rangle_{b}.
\end{multline*}
Take $v=\frac{du}{d\omega} \in L^p((0,T);\{X_t\})$.

(iii) $\Rightarrow$ (iv) clear.

(iv) $\Rightarrow$ (i)
We show that $v=u'$. Let $\varphi\in C_0^{\infty}((0,T))$
and $b=\sup\supp\varphi$. Then for any $e'(b)\in E'_b\subset X'_b$
$$
\langle e'(b), \int_0^T\varphi'u\, d\omega \rangle_{b}
=\int_0^T \varphi'\langle e'(b), u(\omega) \rangle_{b}\, d\omega
= - \int_0^T \varphi\langle e'(b), v(\omega) \rangle_{b}\, d\omega.
$$
Consequently, $\int_0^T\varphi'u\, d\omega = \int_0^T\varphi v\, d\omega$ in $X_b$ as desired.

\end{proof}

We are now ready to formulate and prove our first main result.
\begin{thm}\label{thm:main1}
Let $\{X_t\}$ be a monotone family of reflexive Banach spaces.
If $u\in W^{1,p}((0,T); \{X_t\})$ 
then there exists a section $v(t) \in L^p((0,T); \{X_t\})$ such that $\|v(t)\|_t$ is an upper gradient of $u$ and 
$\|v\|_{L^p((0,T); \{X_t\})} = \inf\limits_{g} \|g\|_{L^{p}((0,T))}$. 
\end{thm}

\begin{proof} 
1. Prove that we can choose a continuous representer $u$.
Let \eqref{eq:criteria} fails on a set $\Sigma\subset (0,T)$ of null measure. 
Observe that thanks to inequality \eqref{eq:criteria} for any convergence sequence 
$\{s_n\}\subset (0,T)\setminus\Sigma$ have 
\begin{equation}\label{eq:criteria:1}
\|u(s_n) - u(s_m) \|_{s_n\vee s_m} \to 0 \quad\text{ when } n,m\to\infty.
\end{equation}
For each point $t \in\Sigma$ choose 
an increasing sequence $\{s_n(t)\}\subset (0,T)\setminus\Sigma$  
such that $\lim\limits_{n\to\infty}s_n(t) = t$.
Due to \eqref{eq:criteria:1}  
$P(s_n(t),t)u(s_n(t))$ is a Cauchy sequence in $X_{t}$, 
and hence there is a limit $\lim\limits_{n\to\infty} P(s_n(t),t)u(s_n(t))\in X_{t}$.
Note that for any other increasing sequence $\{\zeta_n\}\subset (0,T)\setminus\Sigma$
converging to $t$ it is true that $\lim\limits_{n\to\infty}u(\zeta_n)= \lim\limits_{n\to\infty}u(s_n(t))$.
The function 
$$
\bar u(t) :=
\begin{cases}
u(t), & \text{ if } \omega \in (0,T)\setminus\Sigma,\\
\lim\limits_{n\to\infty} P(s_n(t),t)u(s_n(t)), & \text{ if } t \in \Sigma
\end{cases}
$$ 
coincides with $u(t)$ a.~e. on $(0,T)$. 

Now prove that \eqref{eq:criteria} holds for the function $\bar u(t)$
everywhere on $(0,T)$. 
Suppose that $t, t_0 \in (0,T)$ and $t_0 < t$. 
Then $s_n(t_0)<s_n(t)$ for large $n$ and 
\begin{multline*}
\|\bar u(t) - \bar u(t_0)\|_t
\leq
\|\bar u(t) - \bar u(s_n(t))\|_t
+\|\bar u(t_0) - \bar u(s_n(t_0))\|_t\\
+\|\bar u(s(t)) - \bar u(s_n(t_0))\|_t\\
\leq\|\bar u(t) - \bar u(s_n(t))\|_t
+\|\bar u(t_0) - \bar u(s_n(t_0))\|_t\\
+ \int^{s_n(t)}_{s_n(t_0)} g(s)\, ds
\to \int^{t}_{t_0} g(s)\, ds,
\quad \text{ as } n\to\infty.
\end{multline*}
This implies as well the continuity of the function $\bar u(t)$ on $(0,T)$.
Thus we can assume that the function $u(t)$ is continuous and
the inequality \eqref{eq:criteria} is valid everywhere on the interval $(0,T)$.
Due to the continuity, $u((0,T))\cap X_{t_1} = u(\{t\leq t_1\}\cap (0,T))$
is a separable space for any $t_1\in (0,T)$. 
From now we will deal with $\tilde X_{t_1} = u(\{t\leq t_1\}\cap (0,T))$. 
Note that thanks to the reflexivity of $X_{t_1}$,  $\tilde X'_{t_1}$ is separable. 

2. Here we will show that the family of difference quotients is bounded in $L^p$ and uniformly.

For $h>0$ define the function 
$u_h(t) = \frac{u(t) - u(t-h)}{h} \in X_{t}$.
Let $J\subset (0,T)$ and $\operatorname{dist}(J,\{0,T\})>h$. 
Applying \eqref{eq:criteria} and H{\"o}lder's inequality, derive
\begin{equation*}
\|u_h(t)\|^p_{t} =
\frac{1}{h^p}\|u(t) - u(t-h)\|^p_t\\
\leq \frac{1}{h^p}\bigg(\int^{t}_{t - h} g(s)\, ds \bigg)^p
\leq \frac{1}{h}\int^{t}_{t - h} |g(s)|^p\, ds
\end{equation*}
for all $t\in J$.
Next, with the help of Fubini's theorem obtain
 $\|u_h\|_{L^p(J;\{X_t\})} \leq \|g\|_{L^p((0,T))}$,
which means that $\{u_h\}$ is a bounded family in  $L^p(J; \{X_t\})$.

Let $F\subset (0,T)$ be such a set of null measure that
$$
g(t) = \lim\limits_{h\to 0}\frac{1}{h}\int^{t}_{t - h} g(s)\, ds, \quad \text {for all } \omega\in (0,T)\setminus F.
$$
Then the inequality 
$\|u_h(t)\|_t \leq \frac{1}{h}\int^{t}_{t - h} g(s)\, ds$ 
guarantees the uniform estimate
$\|u_h(t)\|_t \leq K_t$ for all $t\in (0,T)\setminus F$
and small $h$.

3. To apply proposition \ref{prop:equivW} (i $\Leftrightarrow$ iv), we will show that sequence $u_h(t)$ has a weak limit in $L^p$, which is the desired derivative. 

Fix $b\in (0,T)$.
Let $\{x'_n(b)\}_{n\in\mathbb N}$ be a dense sequence  in $\tilde X'_b$ 
(it is possible due to step 1).
For $x'_n(b)$ and $t\leq b$ define the function 
$\psi_{n,b}(t) = \langle x'_n(b), u(t) \rangle_{b}$.
Note that
$$
|\psi_{n,b}(t) - \psi_{n,b}(t_0)| \leq \|x'_n(b)\|\int^{t}_{t_0} g(s)\, ds.
$$ 
Therefore the function $\psi_{n,b}(t)$ is locally absolutely continuous on $t\leq b$. 

Since $\tilde X_b$ is reflexive there exists a sequence $h_k\to 0$ and an element
$w(t)\in X_t$ such that
$u_{h_k}(t)\rightharpoonup w(t)$.
Particularly, 
$$
\langle x'_n(b), w(t) \rangle_{b} = \lim\limits_{k\to\infty}\langle x'_n(b), u_{h_k}(t) \rangle_{b}
= \psi'_{n,b}(t)
\quad \text{ for all } t \in (0,T)\setminus F,
$$ 
namely $\psi'_{n,b}(t) = \langle x'_n(b), w(t) \rangle_{b}$ a.~e. on $t\leq b$.

It only remains to verify that $w\in L^p((0,T);\{X_t\})$.
For any sequence $h_m\to 0$ have
$$
\lim\limits_{m\to\infty}\langle x'_n(b), u_{h_m}(t) \rangle_{b} = \psi'_{n,b}(t) = \langle x'_n(b), w(t) \rangle_{b}.
$$
Let $x'(b)\in X'_b$ and $\varepsilon>0$. 
Choose $x'_n(b)$ such that $\|x'(b)-x'_n(b)\|\leq\varepsilon$.
For small $h>0$ infer
\begin{multline*}
|\langle x'(b), u_h(b) - w(b) \rangle_{b}|
\leq
|\langle x'(b) - x'_n(b), u_h(b) - w(b) \rangle_{b}|\\
+ |\langle x'_n(b), u_h(b) - w(b) \rangle_{b}|
\leq
\varepsilon(K_b + \|w(b)\|_b) + \varepsilon.
\end{multline*}
Consequently, $u_{h}(b)\rightharpoonup w(b)$,
and, as number $b$ was chosen arbitrarily,
$u_{h}(\omega)\rightharpoonup w(t)$ for all $t\in (0,T)$.
By proposition \ref{prop:weak-lp} $w\in L^p((0,T); \{X_t\})$.
And by proposition \ref{prop:equivW} there exist a weak derivative $u' = w$.

4. From step 2 $\|\|u'(t)\|_t\|_{L^p((0,T); \mathbb R)} = \|u'\|_{L^p((0,T);\{X_t\})} \leq \|g\|_{L^p((0,T); \mathbb R)}$
for any function $g$ which satisfy \eqref{eq:criteria}.
On the other hand by inequality \eqref{eq:prop:sobolev-int} we have 
$$
\|u(t) - u(t_0)\|_{t} \leq  \int_{t_0}^{t}\|u'(s)\|_{s}\, ds.
$$
Thus $\|u'\|_{L^p((0,T);\{X_t\})} = \inf\limits_{g} \|g\|_{L^{p}((0,T);\mathbb R)}$.
\end{proof}

From the above proof we obtain 
$$
\|u\|_{W^{1,p}((0,T);\{X_t\})} = \|u\|_{L^{p}((0,T);\{X_t\})} + \|u'\|_{L^{p}((0,T);\{X_t\})}.  
$$
Therefore, we may define the Sobolev space $W^{1,p}((0,T);\{X_t\})$ 
as a set of all functions in $L^{p}((0,T);\{X_t\})$ 
with weak derivatives which are also in $L^{p}((0,T);\{X_t\})$.
Now in the usual manner one can proof the following
\begin{thm}\label{theorem:W-Banach}
The space $W^{1,p}((0,T); \{X_t\})$ is a Banach space for all $1\leq p <\infty$.
In the case of Hilbert spaces $X_t$, $W^{1,2}((0,T); \{X_t\})$ is also a Hilbert space.
\end{thm}

\section{Isomorphism between  $W^{1,p}((0,T);\{X_t\})$ and $W^{1,p}((0,T); Y)$.}\label{Isomorphism}
Here we establish requirements to the regularity of the family $\{X_t\}$ which allows us to construct the isomorphism $W^{1,p}((0,T); \{X_t\})$ onto a standard Sobolev space. 
To demonstrate these conditions, we give positive and negative examples.

\subsection{Embeddings between edges}
Let 
$$
\|v\|_T = \lim_{t\to T-}\|v\|_t \quad\text{ and }\quad \|v\|_0 = \lim_{t\to 0+}\|v\|_t,
$$
for $v\in V$. Define $(X_T, \|\cdot\|_T)$ and $(X_0, \|\cdot\|_0)$ as completions
of corresponding  quotient spaces.
Though, in critical cases, the first one could contain only zero, while the last one would be empty.
Then we have the following trivial embeddings:
\begin{prop}\label{prop:trvemb}
The maps 
$$
u(t)\mapsto P(0,t)u(t) \quad \text{ from } W^{1,p}((0,T); X_0) \text{ to } W^{1,p}((0,T); \{X_t\})
$$
$$
u(t)\mapsto P(t,T)u(t) \quad \text{ from } W^{1,p}((0,T); \{X_t\}) \text{ to } W^{1,p}((0,T); X_T)
$$
are both bounded operators.
\end{prop}
\begin{proof}
Let $u \in W^{1,p}((0,T); X_0)$.
Then 
$\|P(0,t)u(t)\|_{t}
\leq \|u(t)\|_{0}
$
and thus $P(0,t) u(t) \in L^{p}((0,T); \{X_t\})$.
By assumption, there is the derivative $u'(t) \in L^p((0,T); X_0)$.
Then 
\begin{multline*}
\|P(0,t) u(t) - P(s,t)P(0,s)u(s)\|_{t} 
\leq \|P(0,t) u(t) - P(0,t)u(s)\|_{t}\\
\leq \|u(t) - u(s)\|_{0}
\leq \int_s^t\|u'(\tau)\|_{0}\, d\tau.  
\end{multline*}

Now let $u\in W^{1,p}((0,T); \{X_t\})$.
Then $\|P(t,T) u(t)\|_{T} \leq \|u(t)\|_{t}$ for all $t\in(0,T)$.
So  $P(t,T) u(t) \in L^{p}((0,T); X_T)$.
Check that there is a derivative in $L^p$
\begin{multline*}
\|P(t,T) u(t) - P(s,T) u(s)\|_{T} = \|P(t,T) u(t) - P(t,T)P(s,t) u(s)\|_{T}\\
\leq \|u(t) - P(s,t) u(s)\|_{T}\\
\leq \|u(t) - P(s,t) u(s)\|_{t}
\leq \int_s^tg(\tau)\, d\tau.  
\end{multline*}
Consequently, $P(t,T) u \in W^{1,p}((0,T); X_0)$.

\end{proof}

\subsection{Isomorphism with a standard Sobolev space}
 
Suppose there is another Banach space $Y$ and a family of isomorphisms 
$\Phi_t:X_t\to Y$ such that 
$$
\|\Phi_tu\|_Y \leq C(t)\|u\|_{X_t} \quad \text{ for all } u\in X_t.
$$
and
$$
\|\Phi^{-1}_tv\|_{X_t} \leq c(t)\|v\|_Y \quad \text{ for all } v\in Y.
$$

\begin{prop}
If 
$\|\Phi_t\|_{\mathcal B(X_t;Y)}$ and $\|\Phi^{-1}_t\|_{\mathcal B(Y;X_t)}$ are in $L^{\infty}((0,T))$
then operator $(\Phi u)(t) = \Phi_tu(t)$ is an isomorphism
from $L^p((0,T);\{X_t\})$ to $L^p((0,T); Y)$.
\end{prop}

We are going to make use of the so-called difference quotient criterion.
This property is well known in the real-valued case \cite[Proposition 9.3]{Brezis11}.
In our settings, we should additionally assume that Banach spaces $X_t$ are isomorphic to some space $Y$ which possesses 
the Radon--Nikodym property. 
A Banach space $Y$ has the Radon--Nikodym property if each Lipschitz continuous function $f:I\to Y$ is differentiable almost everywhere
(a good account on this property see in \cite{ABHN11}).

\begin{prop}\label{prop:crt2}
If there is an isomorphism from $L^p((0,T);\{X_t\})$ to $L^p((0,T); Y)$ and 
$Y$ has the Radon--Nikodym property, then, for $1<p\leq\infty$, a function $u\in L^p((0,T);\{X_t\})$
is in $W^{1,p}((0,T);\{X_t\})$ if and only if there exists a constant $C$ such that 
for all $J\subset\subset(0,T)$ and $h>0$, $\operatorname{dist}(J, \{0,T\}) < h$
$$
\|\tau_h u - u\|_{L^p(J; \{X_{t\vee t+h}\})} \leq Ch.
$$
\end{prop}

The proof of this proposition is a fairly straightforward adaptation of the proof of \cite[Theorem 2.2]{AK2018}.

We want to know when the operator $\Phi:W^{1,p}((0,T); \{X(t)\})\to W^{1,p}((0,T); Y)$
defined by the rule $(\Phi u)(t) = \Phi_tu(t)$
is an isomorphism. Necessary and sufficient conditions are given in the following
\begin{thm}\label{theorem:main2}
Let $1<p\leq\infty$, $N(t,v)=\|v\|_t$ be a continuous function for any $v\in V$, and $Y$ has the Radon--Nikodym property.
Then a family of homeomorphisms $\Phi_t:X_t\to Y$ induces 
an isomorphism $\Phi:W^{1,p}((0,T); \{X_t\})\to W^{1,p}((0,T); Y)$
defined by the rule $(\Phi u)(t) = \Phi_tu(t)$
if and only if $\|\Phi_t\|_{\mathcal B(X_t;Y)}, \|\Phi^{-1}_t\|_{\mathcal B(Y;X_t)} \in L^{\infty}((0,T))$ and 
\begin{equation}\label{theorem:main2:eq1}
\|\Phi_tP(s,t) - \Phi_s\|_{\mathcal B(X_s; Y)} \leq M(t-s)
\end{equation}
\begin{equation}\label{theorem:main2:eq2}
\|\Phi^{-1}_t - P(s,t)\Phi^{-1}_s\|_{\mathcal B(Y; X_t)} \leq M(t-s),
\end{equation}
for almost all $s<t$.
\end{thm}
\begin{proof} 
Necessity.
First we show that $\|\Phi_t\|_{\mathcal B(X_t;Y)}$ are uniformly bounded. 
Fix a cut-off function $\eta\in C^{\infty}_0(\mathbb R)$ equal to $1$ on $B(0, 1)$ 
and $0$ outside the ball $B(0, 2)$. 
By substituting the functions $u_r(t) = \eta(\frac{t-z}{r})v$, 
where $v\in V$ and $B(z,2r)\subset (0,T)$,
into the inequality $\|\Phi u_r\|_{L^p(Y)}\leq K\|u_r\|_{W^{1,p}((0,T);\{X_t\})}$, derive
$$
\int_{|z-t|<r}\|\Phi_t v\|_Y^p\, dt \leq CK \int_{|z-t|<2r}\|v\|_{X_t}^p\, dt.
$$
Applying the Lebesgue differentiation theorem, we infer  
\begin{equation}\label{main2:eq1}
\|\Phi_z v\|_Y\leq C_1\|v\|_{X_z}
\end{equation} 
a.e. on $(0,T)$ for all $v\in V$.

Now note that 
$(\Phi u_r)'(t) = \frac{1}{r}\eta'(\frac{t-z}{r})\Phi_t v + \eta(\frac{t-z}{r})(\Phi_t v)'$.
In the same manner we obtain
\begin{equation}\label{main2:eq2}
\|(\Phi_z v)'\|_Y\leq C_1\|v\|_{X_z}
\end{equation} 
a.e. on $(0,T)$ for all $v\in V$.

Next, from  \eqref{main2:eq2} and the fact that $\Phi\eta v \in W^{1,p}(Y)$, for $v\in V$ and 
$\eta\in C^{\infty}_0(0,T)$, we have
\begin{multline*}
\|(\Phi_tP(s,t) - \Phi_s) v\|_{Y} = \|\Phi_tv - \Phi_s v\|_{Y} 
\leq \int_{s}^{t}\|(\Phi_\tau v)'\|_Y\, d\tau \\
\leq C_1\int_{s}^{t}\|v\|_{X_\tau}\, d\tau 
\leq C_1\|v\|_{X_s} (t-s).
\end{multline*}
For operators $\Phi^{-1}_t$ we repeat the same steps and yield boundedness and estimate \eqref{theorem:main2:eq2}. 

Sufficiency.
Let $u\in W^{1,p}((0,T); \{X(t)\})$ then
\begin{multline*}
\|\Phi_{s+h}u(s+h) - \Phi_su(s)\|_Y \\
\leq\|\Phi_{s+h}(u(s+h) - P(s,s+h)u(s))\|_Y + \|\Phi_{s+h}P(s,s+h)u(s)-\Phi_su(s)\|_Y\\
\leq C(s+h)\|u(s+h) - P(s,s+h)u(s)\|_{X_{s+h}} + \|\Phi_{s+h}P(s,s+h)-\Phi_s\|\cdot\|u(s)\|_{X_s}\\
\leq C(s+h)\int_s^{s+h}g(\tau)\, d\tau + Mh\|u(s)\|_{X_s}.
\end{multline*}
Calculating $L^p$-norm obtain
\begin{multline*}
\|\Phi_{\cdot+h}u(\cdot+h) - \Phi_{\cdot}u(\cdot)\|_{L^p(J; Y)} \\
\leq \bigg( \esssup\limits_{(0,T)}\|\Phi_t\|_{\mathcal B(X_t,Y)}\|g\|_{L^p((0,T))} 
+  M\|u\|_{L^p((0,T), \{X_t\})}\bigg) h.
\end{multline*} 
Due to proposition \ref{prop:crt2} or \cite[Theorem 2.2]{AK2018} $\Phi u\in W^{1,p}((0,T);Y)$. 

In the same manner for operator $(\Phi^{-1} u)(t) = \Phi^{-1}_tu(t)$ we have inequality 
\begin{multline*}
\|\Phi^{-1}_{\cdot+h}u(\cdot+h) - \Phi^{-1}_{\cdot}u(\cdot)\|_{L^p(J; \{X_{t\vee t+h}\})} \\
\leq \bigg( \esssup\limits_{(0,T)}\|\Phi^{-1}_t\|_{\mathcal B(Y,X_t)}\|\|u'\|_{Y}\|_{L^p((0,T))} 
+  M\|u\|_{L^p((0,T), Y)}\bigg) h,
\end{multline*} 
for any $u\in W^{1,p}((0,T); Y)$.
By proposition \ref{prop:crt2} $\Phi^{-1} u\in W^{1,p}((0,T);\{X_t\})$.
\end{proof}

\subsection{Examples}

Let us apply the last theorem to the two examples from section \ref{Sobolev-ineq}.

\begin{example}[Composition operator]

First we examine example \ref{ex-compop}. Since all spaces $X_t$ from this example consists of the same functions as space $W^{1,q}(\Omega_0)$, we want to check if the constructed space $W^{1,p}((0,T); \{X_t\})$ is isomorphic to space $W^{1,p}((0,T); W^{1,q}(\Omega_0))$ under operator $(\Phi u)(t) = w(t)u(t)$ (here $w: (0,T) \to \mathbb{R}$ is some weight function, which we use just for the demonstration of theorem \ref{theorem:main2}).
So we choose operators of multiplication by constant $w(t) I : W^{1,q}(\Omega_t) \to W^{1,q}(\Omega_0)$ as $\Phi_t: X_t \to Y$, $Y = W^{1,q}(\Omega_0)$. 
Then, applying \ref{theorem:main2}, we obtain the conditions: $w$ is Lipschitz function bounded from $0$.

\end{example}

\begin{example}[Monotone family of Hilbert spaces]
As in example \ref{ex-Hilbert}, let $\{Q_t\}$ is a non-decreasing family in $\mathbb R^n$,
and $W^{1,2}((0,T); \{H^1_0(Q_t)\})$ is a Sobolev space in a non-cylindrical domain. 
We want to ask the following question: if is it possible to construct isomorphism $\Phi u (t) = C_{\varphi(t,\cdot)}u(t)$
from this space to a Sobolev space in a cylindrical domain with the help of composition operators between inner spaces
$C_{\varphi(t,\cdot)} : H^1_0(Q_t)\to H^1_0(Q_T)$?
To answer this question we consider a mapping $\varphi(t,x):(0,T)\times Q_T \to \bigcup t\times Q_t$, with the property that for every $t\in(0,T)$ 
mapping $\varphi(t,\cdot):Q_T\to Q_t$ is a quasi-isometry.
Then, for each $t$, operator $C_{\varphi(t,\cdot)}$ is an isomorphism \cite[Theorem 4]{V2005}.

Unfortunately, theorem \ref{theorem:main2} gives us a negative answer to the above question.
The essential obstacle is that assumption \eqref{theorem:main2:eq1} in the case of composition operators is equivalent to the demand on additional derivatives.
To demonstrate this we consider a simple example.
Let $Q_t$ be line segments $[0, 1/2 + t/2]$, where $t\in(0,1)$,
and $\varphi(t,x)=\frac{(1+t)x}{2}$.

\begin{figure}[h!]
\centering	
\includegraphics[width=0.8\linewidth]{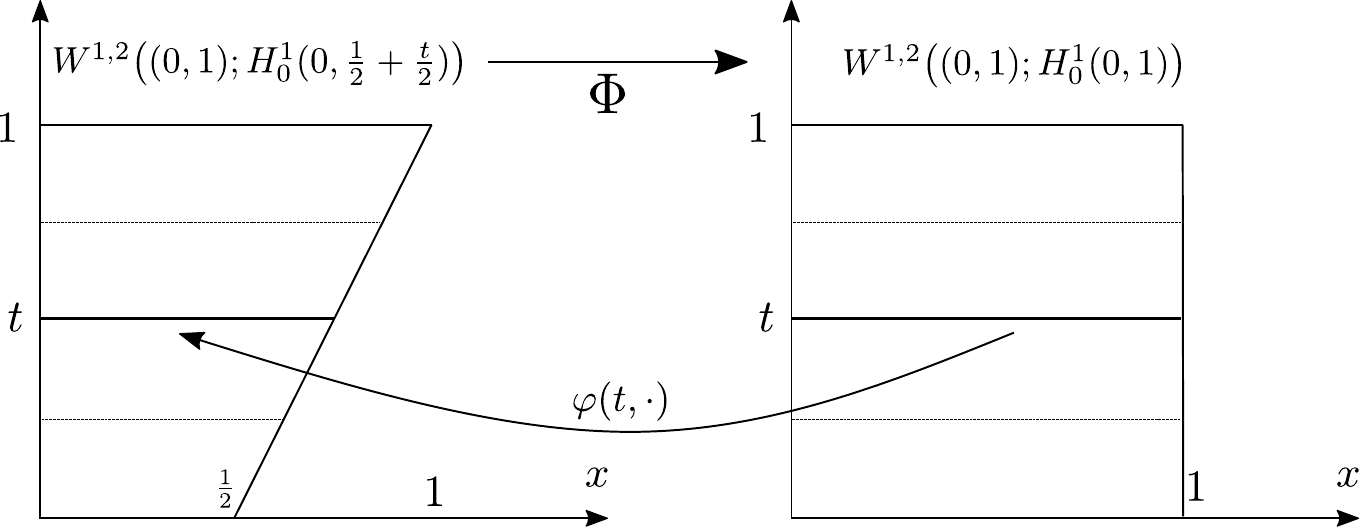}
\caption{} \label{fig:3}
\end{figure}

We will show that \eqref{theorem:main2:eq1} fails.  
To do this we fix $0<s<t<1$ and a point $a\in (0, 1/2+s/2)$.  
Take a sequence $f_n(x) = C|x-a|^{\frac{2}{3}}\eta_0(x)\eta_n(x)$, $n\in\mathbb N$, where:
\begin{itemize} 
\item $\eta_0\in C^\infty_0(Q_s)$ and $\eta_0=1$ on $[\delta, 1/2+s/2-\delta]$,
\item $\eta_n\in C^\infty(Q_s)$ and $\eta_n(x)=0$ for  $x\in [a-1/n, a+1/n]$ and $\eta_n(x)=1$ for  $x\in Q_s\setminus [a-2/n, a+2/n]$,
\item the constant $C$ such that $\|f_n\|_{H^1_0(Q_s)} \leq 1$.  
\end{itemize}
Then $L^2$-norm of $f''_n$ tends to infinity. Using this fact and the Taylor expansion in $1/2(1+s)x$ we obtain 
\begin{multline*}
\frac{1}{t-s}\big\|(C_{\varphi(t,\cdot)}P(s,t) - C_{\varphi(s,\cdot)} ) f_n\big\|_{H^1_0(Q_T)}\\ 
= \frac{1}{t-s}\big\|f_n(1/2(1+t)x) - f_n(1/2(1+s)x)\big\|_{H^1_0(Q_T)}\\
\geq \frac{1}{2}\frac{1}{t-s}\big\|(1+t)f'_n(1/2(1+t)x) - (1+s)f'_n(1/2(1+s)x)\|_{L^2(Q_T)}\\
= \frac{1}{2}\frac{1}{t-s}\big\| f_n''(1/2(1+s)x)\frac{1}{2}(t-s)x + o((t-s)x) \\
+ (t-s)f_n'(1/2(1+s)x) + tf_n''(1/2(1+s)x)\frac{1}{2}(t-s)x + o(t(t-s)x) \big\|_{L^2(Q_T)}\\
= \frac{1}{2}\big\| (1+t)f_n''(1/2(1+s)x)\frac{1}{2}x + f_n'(1/2(1+s)x) +  o(x) \big\|_{L^2(Q_T)} \to \infty \text{ as }  n \to \infty.
\end{multline*}

\end{example}

\section{Scalar characterization}\label{Scalar}
The space $L^{p}((0,T); \{X_t\})$ admits a scalar description:
a measurable section $u(t)$ belongs to $L^{p}((0,T); \{X_t\})$ iff $\|u(t)\|_t\in L^{p}((0,T))$. 
In general case, there is no such characterization for the Sobolev space $W^{1,p}((0,T); \{X_t\})$.
Nevertheless, if $u\in W^{1,p}((0,T); \{X_t\})$ then the norm $\|u(t)\|_t$ enjoys some regularity properties. 
For example, in general settings by rather standard methods one can prove the following inequality:
$$
\|u(t)\|_{L^{\infty}((0,T);\{X_t\})} 
\leq C\|u\|_{1,p} + \|\partial_tN(t,u(t))\|_{L^{p}((0,T); \{X_t\})}
$$
for any  $u\in W^{1,p}((0,T); \{X_t\})$, $1\leq p<\infty$.

We can obtain the scalar characterization in the simplest form when additional conditions are imposed on the norm function. Namely, the following theorem is hold:

\begin{thm}\label{thm:scalar-W}
Let $u\in W^{1,p}((0,T);\{X_t\})$ and assume that\\
(1) For any $v\in V$ function $t\mapsto N(t, v)$ belongs to $W^{1,p}((0,T))$;\\
(2) Sobolev derivatives $\partial_t N(t, v)$ have a majorant $H(t)\in L^p((0,T))$.

Then $\|u(t)\|_t \in W^{1,p}((0,T))$.
\end{thm}
\begin{proof}
$\|u(t)\|_t \in L^{p}((0,T))$ by the definition.

Suppose that $t>s$. Then
\begin{multline*}
\big| \|u(t)\|_t -\|u(s)\|_s \big|
\leq \big| \|u(t)\|_t -\|u(s)\|_t \big| + \big| \|u(s)\|_t -\|u(s)\|_s \big|\\
\leq \|u(t) - u(s)\|_t + \big| N(t, u(s)) - N(s, u(s)) \big|\\
\leq \int^{t}_{s} g_u(\tau)\, d\tau + \int^{t}_{s} |\partial_t N(\tau, u(s))|\, d\tau
\leq \int^{t}_{s} g_u(\tau)+H(\tau)\, d\tau
\end{multline*}
\end{proof}

\begin{cor}
Suppose (1) and (2) from theorem \ref{thm:scalar-W} hold true.
Then there exists a constant $C$ such that
$$
\|u\|_{L^{\infty}((0,T); \{X_t\})} \leq C\|u\|_{W^{1,p}((0,T);\{X_t\})}
$$ 
for any $u\in W^{1,p}((0,T);\{X_t\})$.
\end{cor}

The following theorems gives us another scalar characteristic and establishes a link with the approach of Yu.~G.~Reshetnyak
to define Sobolev spaces of functions with values in a metric space \cite{Reshetnyak97}.

\begin{thm}\label{theorem:Reshetnyak} 
If $u\in L^0((0,T);\{X_t\})$ and the following two assumptions hold:\\
(A) for any $v\in V$ the function $\psi_v(t) = \|u(t) - v\|_{t} \in W^{1,p}((0,T))$,\\
(B) the family of derivatives $\{\psi'_v(t)\}_{v\in V}$ has a majorant $\psi'\in L^p((0,T))$,

then $u\in W^{1,p}((0,T);\{X_t\})$.
\end{thm}
\begin{proof}
1) From (A) when $v=0$ derive $\|u(t)\|_t \in L^p((0,T))$ which implies $u\in L^p((0,T);\{X_t\})$.

2) From (A) and (B) infer that for any $v\in V$
\begin{equation}\label{eq:resh}
\big| \|u(t) - v\|_{t} - \|u(t_0) - v\|_{t_0} \big|
\leq  \int^{t}_{t_0} |\psi'_v(s)|\, ds
\leq  \int^{t}_{t_0} |\psi'(s)|\, ds.
\end{equation}
Let $t\geq t_0$. Choose a sequence $\{v_k\}\subset V$ that
$\|u(t_0) - v_k\|_{t_0}\to 0 $ as $k\to\infty$. 
Then, passage to the limit in \ref{eq:resh}, derive
$$
\|u(t) - u(t_0)\|_{t} \leq  \int^{t}_{t_0} |\psi'(s)|\, ds.
$$
\end{proof}
The converse to \ref{theorem:Reshetnyak} does not hold in general.
The following example shows that 
condition (A) fails
in some cases. 
\begin{example}
Let $V$ be the vector space of all continuous functions defined on $(0,1)$, and $T = 1$. 
Define a family of norms on $V$ 
$$
\|x\|_t
=\begin{cases}
\sup\limits_{(0,1)}|x(s)|, & 0< t<0.5\\
\sup\limits_{(0,1/2)}|x(s)|, & 0.5\leq t< 1,
\end{cases}
$$
where $x(s)\in C((0,1))$. 
Let $D_t$ be a completion of $C(0,1)/\ker \|\cdot\|_t$ with respect to $\|\cdot\|_t$.
Then the function $u(t)(s) = s$ belongs to $W^{1,p}((0,1); \{D_t\})$, $u'(t) = 0$.
On the other hand 
$$
\|u(t)\|_t
=\begin{cases}
1, & 0< t<0.5\\
0.5, & 0.5\leq t< 1,
\end{cases}
$$
is obviously out of space  $W^{1,p}(0,1)$.

\end{example}


\bibliography{direct}

\end{document}